\documentclass[12pt,twoside,reqno]{amsart}
\usepackage[colorlinks=true,citecolor=blue]{hyperref}
\usepackage{mathptmx, amsmath, amssymb, amsfonts, amsthm, mathptmx, enumerate, color}
\usepackage{graphicx}
\usepackage{epstopdf}

\newtheorem{thm}{Theorem}[section]
\newtheorem{cor}{Corollary}[section]
\newtheorem{lm}{Lemma}[section]

\theoremstyle{definition}

\newtheorem{rmk}{Remark}[section]

\numberwithin{equation}{section}

\newcommand{\R}{\mathbb{R}}

\begin{document}
\setcounter{page}{1}

\vspace*{1.0cm}
\title[On the stability of second order parametric (ODE)]
{On the stability of second order parametric ordinary differential equations and applications}
\author[Z. Mazgouri, A. El Ayoubi]{Z. Mazgouri$^{1,*}$, A. El Ayoubi$^2$}
\maketitle
\vspace*{-0.6cm}

\begin{center}
{\footnotesize {\it

$^1$Applied Mathematics Engineering Department, National School of Applied Sciences, Sidi Mohamed Ben Abdellah University, Laboratory LAMA, 30000 Fez, Morocco\\
$^2$Department of Mathematics, Faculty of Sciences, Chouaib Doukkali University, El Jadida Morocco

}}\end{center}

\vskip 4mm {\small\noindent {\bf Abstract.}
This work deals with Lipschitz stability for a parametric version of the general second order Ordinary Differential Equation (ODE) initial-value Cauchy problem. We first establish a Lipschitz stability result for this problem under a partial variation of the data. Then,
we apply our abstract result to second order differential equations governed by cocoercive operators. Furthermore, we discuss more concrete applications of the stability for two specific applied mathematical models inherent in electricity and control theory. Finally, we provide numerical tests based on the software source Scilab, which are done with respect to parametric linear control systems, illustrating henceforth the validity of our abstract theoretical result.

\smallskip


 }

\renewcommand{\thefootnote}{}
\footnotetext{ $^*$Corresponding author.
\par
E-mail addresses: zakariam511@gmail.com (Z. Mazgouri), abdellatifel2014@gmail.com (A. El Ayoubi).
\par
Received ; Accepted }

\section{Introduction}

%\medskip

Throughout this paper, unless otherwise is specified, the Euclidian  space $\mathbb{R}^{n}$ is equipped with the supremum norm $\|.\|.$
%For a given two points $x_0,\, x_{0}^{\prime} \in \mathbb{R}^{n}$ and a nonnegative real number $r>0,$
We denote by $B(x_0,r)$ the open ball of radius $r>0$ centered at $x_0\in \mathbb{R}^{n}$. Let $T>0$ and $\dot{x}_{0}\in \R^n$ and let $f: [0,T]\times B(x_0,r)\times B(\dot{x}_{0},r)\longrightarrow \mathbb{R}^{n}$ be a given  function. We consider the following second order ordinary differential equation (ODE) initial-value Cauchy problem:
%the second order Cauchy ordinary problem
\begin{equation}\label{5}         S(f,x_{0},\dot{x}_{0})  \begin{cases}
\ddot{x}(t)=f(t,x(t),\dot{x}(t)),\,\, \mbox{for}\,\, a.e \,\, t\in [0,T] \\
x(0)=x_{0}\\
\dot{x}(0)=\dot{x}_{0},
\end{cases}
\end{equation}
where $T$ stands for the final time of the interval of interest and $(x_0,\dot{x}_0)\in \R^{2n}$ is the initial-value condition. \\
%For two points $x_0,\, x_{0}^{\prime} \in \mathbb{R}^{n}$ and $T>0$,  given a function $f: [0,T]\times B(x_0,r)\times B(x_{0}^{\prime},r)\longrightarrow \mathbb{R}^{n},$ and $I=[0,T]$ is the interval of interest,
%This structure of ODE is very common
\smallskip

It is well known that the mathematical formulation of second order ODE plays a central role in modelling  a variety of interesting problems inherent in optimal control systems, mechanics and quantum mechanics, dynamical systems subject to classical mechanical laws,  damped simple harmonic motion, the motion of a particle in one dimensional box besides standard concrete applications to electric circuits, the charge of capacitors in related RLC series and many other applied models of a particular relevance well described in the literature as for example in \cite{ALE,Har,Lurie} and the references quoted there.

\smallskip

For these reasons, the classic ODE problems always attract in the widest sense the attention of large community of scientists not only from mathematics and physics or else economic  disciplines but also in all scientific fields as well as engineers and highly qualified technicians. In fact, the literature on the topic available today presents a plethora of contributions which fundamentally regard the existence theory, asymptotic analysis in time and numerical approximation procedures by the use of several  methods such as fixed point approaches, discretization,  regularization procedures, equilibrium techniques and others, see \cite{Braun, Henner et al,Lambert} and the references therein for a detailed presentation on these elements.

\smallskip

A further interesting research direction for ODE problems regards the study of the impact of parametric perturbations on the corresponding solutions. Such parameters could be of an internal type steaming from endogen properties of the considered systems or of an external/exogen one coming from an independent source, which always exposes trajectories of the systems under consideration to be a subject to variations. In this context, the challenge is to find conditions trough which those perturbations will not affect the objective of the systems under consideration. For such purposes, we refer to the very nice contributions of \cite{A,AAM,A.A.J} on the mathematical management of different perturbations within the dynamic feature of the considered differential equation in the same illuminating stability techniques  of \cite{A-Jipam,ABE,AE-AAO,AMR-optim,AMR-JNCA,AR-JCA,AR}.

\smallskip

Let us underline the interest of the quantitative stability analysis for general dynamic problems. It, actually, allows to compute the distance between the solution of the initial system and the one of its perturbed counterpart, which is doubly  meaningful since on the one hand ensures the qualitative convergence of solutions with respect to parameters at any time, and on the other hand provides Lipschitz estimates of trajectories highly required in auxiliary steps arising in the corresponding algorithms for such evolution problems.
%, i.e., the solution of the perturbed system.
Motivated by these numerous factors, the authors of \cite{ALE}, being inspired by \cite{AMR-optim,AMR-JNCA},  presented  quantitative stability results for first order ODE parametric Cauchy problem, which have been proved to be extendable to second order framework as confirmed in the more recent paper \cite{AAM}, wherein the authors considered second order differential equations governed by maximally monotone operators by using two different approaches. In our opinion, such important stability topics still deserve further attention due to the increasing involvement of different old and/or new perturbations in the systems modelled by ODE problems.

%In order to help readers more fully familiar with the, we also refer to the following very recent investigation on the quantitative stability of second \cite{AAM}. \\

%compute the sharpest error bound\\

\smallskip

Accordingly, this paper could be viewed as a continuation of \cite{AAM} and constitutes by the meantime a natural extension of the results in \cite{ALE}. Thus, being very encouraged by the aforementioned motivations and practical applications, we first investigate parametric stability of the system \eqref{5} under perturbation of the function of the second member of the problem and the initial conditions as well.

\smallskip

The perturbation of our purpose involves an external parameter $\lambda$ that belongs to another space different from the one of the state variable of the system. Precisely, $\lambda$ is localized in some open subset $U$ of a some normed space $(\Lambda,\|\,.\,\|)$. In this way, we denote by $\overline{\lambda}\in \Lambda$, the reference value of the parameter $\lambda$, that is $ f_{\overline{\lambda}}=f, \,\,x_{\overline{\lambda}}=x, \,\,x_{\overline{\lambda}}(0)=x_{0}\,\,\, \mbox{and} \,\,\,\dot{x}_{\overline{\lambda}}(0)=\dot{x}_{0}$, and consider a neighborhood $\mathcal{V}(\overline{\lambda})\subset \Lambda$ of $\overline{\lambda}$. Thus, for each $\lambda\in \mathcal{V}(\overline{\lambda})$, the second order parametric Cauchy problem we consider is stated as follows:
\begin{equation}\label{6}
S(f_{\lambda},x_{0,\lambda},\dot{x}_{0,\lambda})  \begin{cases}
\ddot{x}_{\lambda}(t)=f_{\lambda}(t,x_{\lambda}(t),\dot{x}_{\lambda}(t)),\,\, \mbox{for}\,\, a.e \,\, t\in [0,T] \\
x_{\lambda}(0)=x_{0,\lambda}\\
\dot{x}_{\lambda}(0)=\dot{x}_{0,\lambda},
\end{cases}
\end{equation}
where $f_{\lambda}: [0,T]\times B(x_0,r)\times B(\dot{x}_{0},r)\longrightarrow \mathbb{R}^{n} $ is the perturbed function.

\smallskip

In this setting, we seek the conditions under which, for each $\lambda$ near $\overline{\lambda},$ the trajectory $x_{\lambda}:=x(\lambda)$ is Lipschitz continuous, that is the following Lipschitz type estimate holds: for all $\lambda \in \mathcal{V}(\overline{\lambda})$ and for a.e $t\in [0,T]$:
\begin{eqnarray}\label{0}
\|x_\lambda(t)-x(t)\|\leq c_1(t)\|x_{0,\lambda}-x_0\|+c_2(t)\|\dot{x}_{0,\lambda}-\dot{x}_{0}\|+c_3(t)\|\lambda-\overline{\lambda}\|,
\end{eqnarray}
where $c_1,c_2,c_3$ are positive functions to be determined. \\

Our second concern in this paper is the application of our quantitative stability on the abstract formulation of (ODE) in \eqref{6} to concrete examples quoted above. In this sense, our attention is restricted to control systems and electric circuits though our contribution demonstrates indeed the applicability to a large class of similar or other problems that fit in the second order ODE framework.

\smallskip

In what follows we give an outline of the paper. In Section 2, we present the ingredients and tools we need for establishing our main results. In Section 3, based on direct computations with the use of a Perov inequality, we formulate and prove our principal abstract result on Lipschitz stability for the parametric system \eqref{6} (see, Theorem \ref{main-th}). In Subsection 4.1, we first apply this result to second order dynamical systems governed by cocoercive operators (see, Theorem \ref{thm-dyn-syst}). Afterwards, in subsection 4.2, a further interesting applied model is discussed with respect to differential equations modelling RLC circuit's current (see, Theorem \ref{thm-stab-RLC}). Furthermore, in subsection 4.3, based on our main result of Theorem \ref{main-th}, we investigate sensitivity analysis for second order optimal control systems (see, Theorem \ref{appcont}). Finally, in Section 5, we report a numerical example to support and illustrate the theoretical result given in Theorem \ref{appcont}.

\section{Preliminaries}

Let $T>0$ and $x_0\in \R^n$. For a real-valued function $g: [0,T]\times B(x_0,r)\longrightarrow \mathbb{R}^{n}$, we first consider the associated Cauchy problem of initial-value first order ODE which is expressed as follows:
$$ S(g,x_{0})  \begin{cases}
\dot{x}(t)=g(t,x(t)),\,\, \mbox{for}\,\, a.e \,\, t\in [0,T] \\
x(0)=x_{0}.
\end{cases} $$
Let $U$ be an open subset of $\mathbb{R}^{n}$. We recall that a function $g: [0,T]\times U\longrightarrow \mathbb{R}^{n}$ is said to be $L^{1}$-Carath\'eodory if the following conditions are satisfied:
\begin{enumerate}
	\item [(i)] for each $x \in U,$ the map $t\longmapsto g(t,x)$ is measurable;
	\item [(ii)] for a.e $t\in [0,T]$, the map $x\longmapsto g(t,x)$ is continuous;
	\item [(iii)] there exists a function $m \in L^{1}([0,T],\mathbb{R}^{+})$ such that for all $x\in U$ and for a.e $t\in [0,T]$
	\begin{center}
		$\|g(t,x)\|\leq m(t)$.
	\end{center}
\end{enumerate}

The following two results give existence and uniqueness of the solution for a Cauchy problem of the above form in the Carathéodory sense.

\begin{thm}[\text{\cite[Theorem 1]{F}}]\label{1} Let $T>0,$ $x_0\in\mathbb{R}^{n}$ and let $g : [0,T]\times B(x_0,r)\longrightarrow \mathbb{R}^{n}$ be a $L^{1}$-Carathéodory function. Then, for any real number $d$ such that $0<d\leq T$ and $\int_{0}^{d}m(s)ds\leq r,$ the Cauchy problem $S(g,x_{0})$
	%subject to the initial condition $x(0)=x_0$
	admits a unique solution on $[0,d].$
\end{thm}

\begin{thm}[\text{\cite[Theorem 2]{F}}]\label{2} Assume that there exists an integrable function $l$ such that
	$$\|g(t,x)-g(t,y)\| \leq l(t)\|x-y \| \quad  \mbox{for all}\,\, (t,x), (t,y) \in [0,T]\times B(x_0,r).$$
	Then, the Cauchy problem $S(g,x_{0})$
	%$x^{\prime}=f(t,x),\,\,\,x(0)=x_0 $
	admits at most one solution on $[0,T]\times B(x_0,r).$
\end{thm}

%\subsection{The second order Cauchy problem of initial-value ODE}
%Throughout this paper, unless otherwise is specified, the Euclidian  space $\mathbb{R}^{n}$ is equipped with the supremum norm $\|.\|.$
%For a given two points $x_0,\, x_{0}^{\prime} \in \mathbb{R}^{n}$ and a nonnegative real number $r>0,$
%We denote by $B(x_0,r)$ the open ball of radius $r>0$ centered at $x_0\in \mathbb{R}^{n}$. Let $T>0$ and $x_{0}^{\prime}\in \R^n$,

Let us now return to the problem \eqref{5} subject to our treatment in the present contribution.  We make use of the change of variable
$u:= (x, \dot{x})$ and define the corresponding function $\tilde{f}(t,x):=(\dot{x},f(t,x,\dot{x}))$.
%Clearly, the function $x$ is a solution to the problem $S(f,x_{0},x_{0}^{\prime})$ if and only if $(x,x^{\prime})$ is a solution to the following Cauchy problem of initial-value first order ODE:
With this change of variable, by augmenting the dimension of the state variable, the second order
ODE \eqref{5} becomes a first order ODE as follows:
\begin{equation}     S(\tilde{f},u_{0}) \; \begin{cases}
\dot{u}(t)=\tilde{f}(t,u(t)),\,\, \mbox{for}\,\, a.e \,\, t\in [0,T],\,\, u\in \mathbb{R}^{2n} \\
u(0)=u_0.
\end{cases}
\end{equation}
%where $u=(x,y), \tilde{f}(t,u)=(y,f(t,x,y))$ and $u_0=(x_0,x_0^{\prime})$. \\
%\bigskip

%\subsection{Carathéodory conditons and existence}
\noindent Consequently, the next result is straight from Theorem \ref{1}.
\begin{cor}
	Let $f: [0,T]\times B(x_0,r)\times B(\dot{x}_{0},r)\longrightarrow \mathbb{R}^{n}$ be a function that satisfies the assumptions below:
	%the Carathéodory conditions, i.e.,
	\begin{enumerate}
		\item for every $(x,y)\in \mathbb{R}^{2n}$, $f(.,x,y)$ is measurable;
		\item for a.e $t\in [0,T]$, $f(t,.,.)$ is continuous;
		\item for every $r>0$ there exists $g_{r}\in L^{1}([0,T],\mathbb{R}^{+})$ such that whenever $\|x-x_0\|\leq r$ and $\|y-\dot{x}_0\|\leq r$, it holds
		\begin{center}
			$\|f(t,x,y)\|\leq g_{r}(t)$, \, for a.e $t\in [0,T]$.
		\end{center}
	\end{enumerate}
	Then, for any real number $d$ such that $0<d\leq T$ and $\int_{0}^{d}g_r(s)ds\leq r,$ the Cauchy problem \eqref{5}
	%subject to the initial condition $x(0)=x_0$
	admits a unique solution on $[0,d].$
\end{cor}
%$T$ being a nonnegative real number standing for the final time of the interval of interest.
%Then, the corresponding Cauchy problem of initial-value ODE associated with these data is as follows:
%\begin{equation}\label{5}   S(f,x_{0},x_{0}^{\prime})  \; \begin{cases}
%                      x^{\,\prime\prime}(t)=f(t,x(t),x^{\prime}(t)),\,\, \mbox{for}\,\, a.e \,\, t\in [0,T] \\
%                      x(0)=x_{0}\\
%                      x^{\prime}(0)=x_{0}^{\prime}.
%                   \end{cases}
%\end{equation}

We also need the first part of Theorem 1 on page 360 of the book \cite{M.P.F}, which is the statement of Perov's inequality.

\begin{thm}[\text{\cite[Theorem 1]{M.P.F}}]\label{3}
	Let $u(.)$ be a nonnegative function that satisfies the integral inequality
	\begin{equation}\label{}
	u(t)\leq c+\int_{t_0}^{t}(a(s)u(s)+b(s)u^{\alpha}(s))ds,\quad c,\alpha \geq 0,
	\end{equation}
	where $a(t)$ and $b(t)$ are continuous nonnegative functions for $t\geq t_0$.\\
	Then, for $0\leq \alpha <1$ we have
	\begin{equation}\label{}
	u(t)\leq \left\{c^{1-\alpha}e^{(1-\alpha)\int_{t_0}^{t}a(s)ds}+(1-\alpha)\int_{t_0}^{t}b(s)e^{(1-\alpha)\int_{s}^{t}a(r)dr}ds\right\}^\frac{1}{1-\alpha}.
	\end{equation}
\end{thm}

We close this section with the following lemma that will be needed to prove our main result of the next section.
%in Theorem \ref{main-th} below.
\begin{lm}\label{lm1}
	Let $f\in C^{2}([0,T],\mathbb{R}^{n})$
	%$f:[0,T]\longrightarrow \mathbb{R}^{n}$ in
	such that $f(0)=0$.  Then, for each $t\in [0,T]$, we have
	$$ \int_{0}^{t}\|f(s)\|\|\dot{f}(s)\|ds \leq \frac{t}{2}\int_{0}^{t}\|\dot{f}(s)\|^{2}ds.$$
\end{lm}

\begin{proof}
Consider the function $g$ defined on $[0,T]$ by
%\begin{center}
	$g(t)=\int_{0}^{t}\|\dot{f}(u)\|du$. \\
%\end{center}
We have for every $t\in [0,T]$, $\dot{g}(t)= \|\dot{f}(t)\|$. Fix $t\in [0,T]$, thus for $s \in [0,t]$, we have  $$\|f(s)\|\|\dot{f}(s)\|\leq g(s)\dot{g}(s).$$
Integrating this inequality in $[0,t]$ and using the Cauchy Schwarz inequality, immediately follows
$$ \int_{0}^{t}\|f(s)\|\|\dot{f}(s)\|ds \leq \int_{0}^{t}g(s)\dot{g}(s)ds = \frac{1}{2}g^{2}(t)\leq \frac{t}{2}\int_{0}^{t}\|\dot{f}(s)\| ^{2}ds.$$
This completes the proof.
\end{proof}

\section{Main result}

In this section, we consider the perturbed format of the system $S(f, x_{0},\dot{x}_{0})$ given by \eqref{6}. Based on a Perov's argument, we provide conditions under which the Lipschitz estimate in \eqref{0} is satisfied.
%which involves an external parameter $\lambda$ that belongs to another space. Precisely, $\lambda$ is localized in some open subset $U$ of a some normed space $(\Lambda,\|\,.\,\|)$. In this way, the second order parametric Cauchy problem we consider is stated as follows:
%\begin{equation}\label{6}
%                  S(f_{\lambda},x_{0,\lambda},x_{0,\lambda}^{\prime})  \begin{cases}
%                     x_{\lambda}^{\,\prime\prime}(t)=f_{\lambda}(t,x_{\lambda}(t),x_{\lambda}^{\prime}(t)),\,\, \mbox{for}\,\, a.e \,\, t\in [0,T] \\
%                     x_{\lambda}(0)=x_{0,\lambda}\\
%                    x_{\lambda}^{\prime}(0)=x_{0,\lambda}^{\prime},
%                 \end{cases}
%\end{equation}
%where $f_{\lambda}: [0,T]\times B(x_0,r)\times B(x_{0}^{\prime},r)\longrightarrow \mathbb{R}^{n} $. The initial value of the parameter $\lambda$ is denoted by $\overline{\lambda}$, i.e., $ f_{\overline{\lambda}}=f, \,\,x_{\overline{\lambda}}=x, \, \,\,x_{\overline{\lambda}}(0)=x_{0}\,\,\, \mbox{and} \,\,\,x^{\prime}_{\overline{\lambda}}(0)=x^{\prime}_{0}.$\\

\medskip

\begin{thm}\label{main-th}
	Assume that for some $L > 0,\,\, L^{\prime} > 0$ the following conditions hold:
	\begin{enumerate}
		\item the function $f_\lambda$ is continuous;
		\item the function $f_\lambda$ is $L$-Lipschitz in $(x,y)$ uniformly in $t$ and $\lambda$, that is
		$$  \|f_{\lambda}(t,x,y)-f_{\lambda}(t,x^{\prime},y^{\prime})\| \leq L(\|x-x^{\prime}\|+\|y-y^{\prime}\|);$$
		\item the function $f_\lambda$ is $L^{\prime}$-Lipschitz in $\lambda$ uniformly in $t$ and $(x,y)$, that is
		$$ \|f_{\lambda}(t,x,y)-f_{\lambda^{\prime}}(t,x,y)\| \leq L^{\prime}\|\lambda-\lambda^{\prime}\|.$$
	\end{enumerate}
	Let $x_\lambda$ $($resp., $x)$ be the solution of the problem $S(f_{\lambda},x_{0,\lambda},\dot{x}_{0,\lambda})$ $($resp., $S(f,x_{0},\dot{x}_{0}))$. Then, for each $\lambda \in \mathcal{V}(\overline{\lambda})$ and for a.e $t\in [0,T]$, we have
	\begin{equation}\label{estip}
	\begin{split}
	&\|x_{\lambda}(t)-x(t)\|\leq  \left(1+L\left(\frac{2}{2+LT}\right)^{2}\left(e^{\frac{(2+LT)t}{2}}-1-t\right)\right)\|x_{0,\lambda}-x_{0}\| \\
	&\quad\quad+\frac{2}{2+LT}\left(e^{\frac{(2+LT)t}{2}}-1\right)\|\dot{x}_{0,\lambda}-\dot{x}_{0}\|+
	L^{\prime}\left(\frac{2}{2+LT}\right)^{2}\left(e^{\frac{(2+LT)t}{2}}-1-t\right)\|\lambda-\overline{\lambda}\|.
	\end{split}
	\end{equation}
\end{thm}

\begin{proof}
Consider the function $ \phi_{\lambda}(t)=\|\dot{x}_{\lambda}(t)-\dot{x}(t)\|^{2}$. We have
$$\dot{\phi}_{\lambda}(t)= 2\langle \dot{x}_{\lambda}(t)-\dot{x}(t),\ddot{x}_{\lambda}(t)-\ddot{x}(t) \rangle.$$
Using Cauchy Schwarz and the triangular inequality we obtain
%$$
% \phi^{\prime}(t)\leq  2\|x_{\lambda}^{\prime}(t)-x^{\prime}(t)\|.\|f(t,x(t),x^{\prime}(t))-f_{\lambda}(t,x_{\lambda}(t),x_{\lambda}^{\prime}(t))\|.
%$$
%Thus,
$$ \dot{\phi}_{\lambda}(t)  \leq 2\|\dot{x}_{\lambda}(t)-\dot{x}(t)\|\left(\|f(t,x(t),\dot{x}(t))-f(t,x_{\lambda}(t),\dot{x}_{\lambda}(t))\|+
\|f(t,x_{\lambda}(t),\dot{x}_{\lambda}(t))-
f_{\lambda}(t,x_{\lambda}(t),\dot{x}_{\lambda}(t))\|\right).  $$
Since $f_\lambda$ is $L$-Lipschitz in $(x,y)$ uniformly in $t$ and $\lambda$ and  $L^{\prime}$-Lipschitz in $\lambda$ uniformly in $t$ and $(x,y)$, we get
\begin{equation*}\label{13}
\dot{\phi}_{\lambda}(t)\leq 2L\phi_{\lambda}(t)+2L\|\dot{x}_{\lambda}(t)-\dot{x}(t)\|\|x_{\lambda}(t)-x(t)\|+2L^{\prime}\|\lambda-\overline{\lambda}\|\|\dot{x}_{\lambda}(t)-\dot{x}(t)\|.
\end{equation*}
On the other hand, we have
\begin{equation}\label{14}
\|x_{\lambda}(t)-x(t)\|\leq \int_{0}^{t}\|\dot{x}_{\lambda}(s)-\dot{x}(s)\|ds+\|x_{\lambda}(0)-x(0)\|.
\end{equation}
By combining the last two inequalities we obtain
\begin{equation}\label{15}
\dot{\phi}_{\lambda}(t)\leq 2\phi_{\lambda}(t)+2L(\phi_{\lambda}(t))^{\frac{1}{2}}\int_{0}^{t}(\phi_{\lambda}(s))^{\frac{1}{2}}ds
+\left(2L\|x_{\lambda}(0)-x(0)\|+2L^{\prime}\|\lambda-\overline{\lambda}\|\right)(\phi_{\lambda}(t))^{\frac{1}{2}}.
\end{equation}
Applying Lemma \ref{lm1}, we obtain
\begin{equation*}\label{16}
\int_{0}^{t}\left( (\phi_{\lambda}(s))^{\frac{1}{2}}\int_{0}^{s}(\phi_{\lambda}(u))^{\frac{1}{2}}du\right)ds \leq \frac{t}{2} \int_{0}^{t} \phi_{\lambda}(s)ds.
\end{equation*}
Integrating the above inequality from $0$ to $t$ and using the inequality \eqref{15} yields
$$ \phi_{\lambda}(t)\leq \phi_{\lambda}(0)+2\int_{0}^{t}\phi_{\lambda}(s)ds+Lt\int_{0}^{t}\phi_{\lambda}(s)ds+\left(2L\|x_{\lambda}(0)-x(0)\|+2L^{\prime}|\lambda-\overline{\lambda}\|\right)\int_{0}^{t}(\phi_{\lambda}(s))^{\frac{1}{2}}ds. $$
Thus, for all $t\in [0,T]$
$$
\phi_{\lambda}(t)\leq \phi_{\lambda}(0)+(2+LT)\int_{0}^{t}\phi_{\lambda}(s)ds+\left(2L\|x_{\lambda}(0)-x(0)\|+2L^{\prime}\|\lambda-\overline{\lambda}\|\right)\int_{0}^{t}(\phi_{\lambda}(s))^{\frac{1}{2}}ds. $$
Applying Theorem \ref{3} with \begin{center}
	$t_0=0,\alpha=\frac{1}{2}, c=\phi_{\lambda}(0), a(.)=2+LT, b(.)=2L\|x_{\lambda}(0)-x(0)\|+2L^{\prime}\|\lambda-\overline{\lambda}\|$ and $u(.)=\phi_{\lambda}(.)$,\end{center}
we see that
$$ \phi_{\lambda}(t) \leq \left\{ \phi_{\lambda}(0)^{\frac{1}{2}}e^{\frac{1}{2}\int_{0}^{t}(2+LT)ds}+\left(L\|x_{\lambda}(0)-x(0)\|+L^{\prime}\|\lambda-\overline{\lambda}\|\right)\int_{0}^{t}e^{\frac{1}{2}\int_{s}^{t}(2+LT)dr}ds\right\}^{2}.
$$
Thus, %since $0\leq t \leq 1$, we have
\begin{equation}\label{xlambda'} \|\dot{x}_{\lambda}(t)-\dot{x}(t)\|\leq  \|\dot{x}_{\lambda}(0)-\dot{x}(0)\|e^{\frac{(2+LT)t}{2}}+2\frac{L\|x_{\lambda}(0)-x(0)\|+L^{\prime}\|\lambda-\overline{\lambda}\|}{2+LT}( e^{\frac{(2+LT)t}{2}}-1).
\end{equation}
Integrating this inequality from $0$ to $t$ and using the inequality (\ref{14}), we obtain
\begin{equation*}
\begin{split}
\|x_{\lambda}(t)-x(t)\|&\leq  \left(1+L\left(\frac{2}{2+LT}\right)^{2}\left(e^{\frac{(2+LT)t}{2}}-1-t\right)\right)\|x_{\lambda}(0)-x(0)\| \\
&\quad+\frac{2}{2+LT}\left(e^{\frac{(2+LT)t}{2}}-1\right)\|\dot{x}_{\lambda}(0)-\dot{x}(0)\|+
L^{\prime}\left(\frac{2}{2+LT}\right)^{2}\left(e^{\frac{(2+LT)t}{2}}-1-t\right)\|\lambda-\overline{\lambda}\|.
\end{split}
\end{equation*}
The proof is hence complete.
\end{proof}

\section{Applications}

\subsection{Second order dynamical systems governed by cocoercive operators}
The object of this subsection is to study the sensitivity analysis for a parametric second-order dynamical system governed by a given maximal monotone operator $A: H\rightarrow H$ defined on a Hilbert space $H$. For $T>0$, the problem under consideration is stated as follows:

\begin{equation}\label{init-syst}
\tag*{$S(A,u_0,v_0)$}
\begin{cases}
\;\ddot{x}(t)+\gamma \dot{x}(t)+Ax(t)=0,\;\;\mbox{for}\,\, a.e \,\, t\in [0,T] \\
\;x(0)=u_0 \hbox{ and }\;\dot{x}(0)=v_0,
\end{cases}
\end{equation}
where $\gamma>0$ is a damping coefficient and $u_0, v_0\in H$ are initial conditions. \\
We consider here the notion of strong solution of \ref{init-syst}, that is  $x:[0, T]\rightarrow H$ such that $x$ and $\dot{x}$ are absolutely continuous on the interval $[0,T]$, $x(0)=u_0$, $\dot{x}(0)=v_0$ and $\ddot{x}(t)+\gamma \dot{x}(t)+Ax(t)=0$ for almost every $t\in [0, T]$. For interesting particular cases which motivate the above dynamical system, we refer to \cite{AAM, Alv-Att-1, Alv-Att-2} and the references therein.
% Since the $\alpha$-cocoercivity of $A$ obviously implies that $A$ is $\frac{1}{\alpha}$-Lipschitz continuous, the existence and uniqueness of a strong solution $x\in C^2([0, +\infty);H)$ of \ref{init-syst} can be shown via the classical Cauchy-Lipschitz-Picard theorem (see, \cite{Haraux}) by rewriting \ref{init-syst} as a first order dynamical system in the product space $H\times H$. \\
In our concern, under perturbation of both the operator $A$ and the initial values, the parametric form of the above system is stated as follows:

\begin{equation}\label{coco-syst}
\tag*{$S(A_{\lambda},u_{0,\lambda},v_{0,\lambda})$} \quad \begin{cases}
\ddot{x}_{\lambda}(t)+\gamma  \dot{x}_{\lambda}(t)+A_{\lambda}x_{\lambda}(t)=0,\;\;\mbox{for}\,\, a.e \,\, t\in [0,T] \\
x_{\lambda}(0)=u_{0,\lambda}\,\,\, \mbox{and}\,\,\, \dot{x}_{\lambda}(0)=v_{0,\lambda}.
\end{cases}
\end{equation}

\noindent In our sensitivity analysis of \ref{coco-syst}, the following two hypotheses will be considered:
\begin{itemize}
	\item[$(C_1)$] {\bf{Cocoercivity assumption}}: for each $\lambda \in \mathcal{V}(\overline{\lambda}),$ there exists $\alpha_{\lambda}>0$ such that $\forall u,v \in H$
	$$ \langle A_{\lambda}(u)-A_{\lambda}(v),u-v \rangle \geq \alpha_{\lambda}\|A_{\lambda}(u)-A_{\lambda}(v)\|^{2};$$
	\item[$(C_2)$] {\bf{Lipschitz property with respect to parameter $\lambda$}}: there exists $L^{\prime} >0$ such that for all $u \in H$ and all $\lambda,\lambda^{\prime} \in \mathcal{V}(\overline{\lambda}),$
	$$\|A_{\lambda}(u)-A_{\lambda^{\prime}}(u)\|\leq L^{\prime}\|\lambda-\lambda^{\prime}\|.$$
\end{itemize}

\begin{thm}\label{thm-dyn-syst}
	Assume that conditions $(C_1)-(C_2)$ are satisfied. Then, for each $\lambda \in \mathcal{V}(\overline{\lambda})$, the problem \ref{coco-syst} admits a unique solution $x_{\lambda}.$ Furthermore, if $\frac{1}{\alpha_{\lambda}}< \gamma $ for all $\lambda \in \mathcal{V}(\overline{\lambda})$, then for a.e $t\in [0,T]$, the following Lipschitz estimate holds:
	\begin{equation}\label{estimate-sods}
	\begin{split}
	&\|x_{\lambda}(t)-x(t)\|\leq  \left(1+\gamma\left(\frac{2}{2+\gamma T}\right)^{2}\left(e^{\frac{(2+\gamma T)t}{2}}-1-t\right)\right)\|u_{0,\lambda}-u_{0}\| \\
	&\quad+\frac{2}{2+\gamma T}\left(e^{\frac{(2+\gamma T)t}{2}}-1\right)\|v_{0,\lambda}-v_{0}\|+
	L^{\prime}\left(\frac{2}{2+\gamma T}\right)^{2}\left(e^{\frac{(2+\gamma T)t}{2}}-1-t\right)\|\lambda-\overline{\lambda}\|.
	\end{split}
	\end{equation}
\end{thm}

\begin{proof}
Clearly, the formulation in the problem \ref{coco-syst} may be expressed as \eqref{5} by considering the function $f_{\lambda}$ defined, for all $\lambda \in \mathcal{V}(\overline{\lambda})$ and all $x,y\in H$, by $f_{\lambda}(t,x,y)=-A_{\lambda}x-\gamma y$.
Owing to the fact that the $\alpha_{\lambda}$-cocoercivity of $A_{\lambda}$ in $(C_1)$ implies that $A_{\lambda}$ is $\frac{1}{\alpha_{\lambda}}$-Lipschitz continuous, the existence and uniqueness of a strong solution to \ref{coco-syst} can be shown via the classical Cauchy-Lipschitz-Picard theorem (see, \cite{Haraux}) by rewriting this system as a first order dynamical system in the product space $H\times H$.

Now, for proving the above estimate, we will make use of Theorem \ref{main-th}. Combining the $\frac{1}{\alpha_{\lambda}}$-Lipschitz continuity of $A_{\lambda}$ with the hypothesis $\frac{1}{\alpha_{\lambda}}< \gamma$ for all $\lambda \in \mathcal{V}(\overline{\lambda})$, one can conclude that the function $f_{\lambda}$ is $\gamma$-Lipschitz in $(x,y)$ uniformly in $t$ and $\lambda$.
On the other hand, it is easy to check that the condition $(C_2)$ implies that $f_\lambda$ is $L^{\prime}$-Lipschitz in $\lambda$ uniformly in $t$ and $(x,y).$ Consequently, assumptions of Theorem \ref{main-th} are all satisfied. Applying this theorem, we immediately obtain  the required estimate.
\end{proof}

\begin{rmk}
	\begin{enumerate}
		\item For some instances for which the assumptions $(C_1)$ and $(C_2)$ are verified, we refer to \cite[Example 3.3]{AAM}.
		\item If the initial conditions are not subject to perturbation, (i.e., if $(u_{0,\lambda},v_{0,\lambda})=(u_0,v_0)$ for all $\lambda \in \mathcal{V}(\overline{\lambda})$), then the
		estimation of Theorem \ref{thm-dyn-syst} reduces to: for all $\lambda \in
		\mathcal{V}(\overline{\lambda})$ and for a.e. $t\in [0,T]$,	
		$$\|x_{\lambda}(t)-x(t)\|\leq 
		L^{\prime}\left(\frac{2}{2+\gamma T}\right)^{2}\left(e^{\frac{(2+\gamma T)t}{2}}-1-t\right)\|\lambda-\overline{\lambda}\|.$$
		\item In \cite[Theorem 3.1]{AAM}, based on a Gronwall-Bellman argument, the authors obtained an estimate similar to the one of our Theorem \ref{thm-dyn-syst} with same hypotheses by reformulating the second order dynamical system \ref{coco-syst} to a first order one. While in our case here, thanks to the Perov's inequality, the estimate \eqref{estimate-sods} is rather obtained by a direct computation without recourse to the first order case. 
	\end{enumerate}
\end{rmk}

\subsection{Differential equation of RLC circuit's current}

In this paragraph, we apply our main result on Lipschitz stability of the previous section to a concrete problem mathematically modelled as an initial-value problem for second order ODE, which interests in particular the audience from physics and engineering. In this respect, we consider the following equation modelling the electric current in an RLC parallel circuit, also known as a tuning circuit (see \cite{Har}):
%In \cite[Theorem 4.1]{I-R},
%the authors prove that under certain conditions
\begin{equation}\label{ex1}
\begin{cases}
\frac{d^{2}x}{dt^{2}}+\frac{R}{L}\frac{dx}{dt}=g(t,x(t)),\,\,\, t\in [0,1] \\
x(0)=\dot{x}(0)=0,
\end{cases}
\end{equation}
where $g:[0,1]\times \mathbb{R}^{+}\longrightarrow \mathbb{R}$ is a continuous function, $R$ is the resistance of the resistor and $L$ is the inductance of the inductor.
%$\tau=\frac{R}{L}$ is a constant characterising the system. Here, $R$ is the resistance of the resistor and $L$ is the inductance of the inductor.
%%%%%%%%%%%%%%%
\begin{flushleft}
	\begin{figure}[h]
		\vspace{-0.6cm}
		\begin{center}
			\includegraphics[width=5cm]{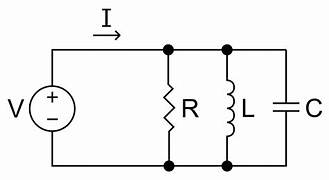}
			\vspace{-0.2cm}
			\caption{\small RLC parallel circuit}
			\label{RLC}
		\end{center}
	\end{figure}
\end{flushleft}
Clearly, the problem (\ref{ex1}) is equivalent to the following integral equation
\begin{equation}\label{ex2}
x(t)=\int_{0}^{1}G(t,s)g(s,x(s))ds,\,\, t\in [0,1],
\end{equation}
where $G$ is the Green's function defined by
\begin{equation}\label{ex3}
G(t,s)=\begin{cases}
\frac{1}{\tau}(1-e^{\tau(s-t)})   &  if \quad 0\leq s\leq t\leq 1 \\
0 \qquad \qquad  &  if  \quad   0\leq t\leq s \leq 1,
\end{cases}
\end{equation}
and $\tau= \frac{R}{L}$. 
Following the same approach used in the proof of \cite[Theorem 4.1]{I-R},
%we prove by correcting the Green's function proposed by the authors and replacing it by
%\eqref{ex3}, that
the problem (\ref{ex2}) admits a unique solution under the following
%new
hypotheses:
\begin{itemize}
	\item there exists a continuous function $w:[0,1]\longrightarrow (0,\infty)$   such that
	$$ |g(s,x(s))| \leq  \frac{\tau}{2}w(s) |x(s)|, $$ for all $s\in [0,1]$;
	\item $\max_{s\in [0,1]}w(s)=e^{-\alpha}$, where $\alpha>e$.
\end{itemize}
In our concern, we consider that the function $g$ is also subject to perturbation, the perturbed circuit is modelled by the following equation:
\begin{equation}\label{pbb}
\begin{cases}
\frac{d^{2}x_\lambda}{dt^{2}}+\frac{R}{L}\frac{dx_\lambda}{dt}=g_{\lambda}(t,x_\lambda(t)),\,\, \mbox{for}\,\, a.e \,\, t\in [0,1] \\
x_{\lambda}(0)=x_{0,\lambda}\\
\dot{x}_{\lambda}(0)=\dot{x}_{0,\lambda}.
\end{cases}
\end{equation}
The initial value of the parameter $\lambda$ is denoted by $\overline{\lambda}$, that is, $g_{\overline{\lambda}}=g, x_{\overline{\lambda}}=x, x_{\overline{\lambda}}(0)=x_{0}\, \mbox{and} \,\dot{x}_{\overline{\lambda}}(0)=\dot{x}_{0}.$

\medskip
%The next result can be seen as a refinement of \cite[Theorem 4.1]{I-R} since the second member of the hypothesis does not depend on $\tau$, consequently,
As a direct application of Theorem \ref{main-th}, we obtain the following parametric stability result for the system \eqref{pbb}.
\begin{thm}\label{thm-stab-RLC}
	Suppose that the following assumptions are satisfied:
	\begin{enumerate}
		\item [(i)] the function $g_\lambda$ is continuous;
		\item [(ii)] for some $\beta>0$, the function $g_\lambda$ is $\beta$-Lipschitz in $x$ uniformly in $t$ and $\lambda$;
		\item [(iii)] for some $L^{\prime}>0$, the function $g_\lambda$ is $L^{\prime}$-Lipschitz in $\lambda$ uniformly in $t$ and $x$;
		\item [(iv)] there exists a continuous function $w:[0,1]\longrightarrow (0,\infty)$  such that
		$$ |g_{\lambda}(s,x(s))| \leq \frac{\tau}{2} w(s) |x(s)|, $$ for all $s\in [0,1]$ and all $\lambda \in \mathcal{V}(\overline{\lambda})$;
		\item [(v)] $\max_{s\in [0,1]}w(s)=e^{-\alpha}$, where $\alpha>e$.
	\end{enumerate}
	Then, for each $\lambda \in \mathcal{V}(\overline{\lambda})$, the problem \eqref{pbb} admits at least a solution and the following Lipschitz estimate holds: for a.e $t\in [0,1]$
	\begin{equation}\label{stab-RLC}
	\begin{split}
	\|x_{\lambda}(t)-x(t)\|\leq& \left(1+L\left(\frac{2}{2+L}\right)^{2}\left(e^{\frac{(2+L)t}{2}}-1-t\right)\right)\|x_{0,\lambda}\| \\
	&+\frac{2}{2+L}\left(e^{\frac{(2+L)t}{2}}-1\right)\|\dot{x}_{0,\lambda}\|+
	L^{\prime}\left(\frac{2}{2+L}\right)^{2}\left(e^{\frac{(2+L)t}{2}}-1-t\right)\|\lambda-\overline{\lambda}\|,
	\end{split}
	\end{equation}
	where $L=\sup(\beta,\tau)$.
\end{thm}

\begin{proof}
Under assumptions $(iv)$ and $(v)$, the existence of a solution to the problem (\ref{pbb}) follows from \cite[Theorem 4.1]{I-R}. Now, to prove the Lipschitz estimate \eqref{stab-RLC}, it suffices to apply Theorem \ref{main-th}. To this end, let us consider the function $f_{\lambda}:[0,1]\times \mathbb{R}^{+}\times \mathbb{R}^{+}\longrightarrow \mathbb{R}$ defined, for each $\lambda\in \mathcal{V}(\overline{\lambda})$, by
$$ f_{\lambda}(t,x,y)= g_{\lambda}(t,x)-\tau y.$$
Using hypothesis $(ii)$, we see that the function $f_{\lambda}$ is $L=\sup(\beta,\tau )$-Lipschitz in $(x,y)$ uniformly in $t$.
On the other hand, assumption $(iii)$ implies that $f_{\lambda}$ is $L^{\prime}$-Lipschitz in $\lambda$ uniformly in $(x,y)$. Consequently, the desired estimate \eqref{stab-RLC} immediately follows from inequality \eqref{estip} with $T=1$ and $x_0=\dot{x}_0=0$. 
%This completes the proof.
\end{proof}

%\newpage
\begin{rmk}
	Our approach is of course also applicable to an RLC series circuit of the following type:
	%%%%%%%%%%%%%%%
	\begin{flushleft}
		\begin{figure}[h]
			\vspace{-0.2cm}
			\begin{center}
				\includegraphics[width=5cm]{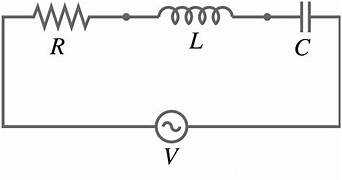}
				\vspace{-0.4cm}
				\caption{\small A series RLC circuit}
				\label{}
			\end{center}
		\end{figure}
	\end{flushleft}
	Indeed, since the charge on the capacitor in such circuit can be modelled by the following second order ODE:
	$$\frac{d^{2}q(t)}{dt^{2}}+\frac{R}{L}\frac{dq(t)}{dt}+\frac{1}{LC}q(t)=V(t),$$
	where, $L$ is the inductance, $R$ is the resistance, $C$ is the capacitance and $V(t)$ is the voltage source. \\
	We easily see that the above equation can be equivalently written as $$\ddot{q}(t)=f(t,q(t),\dot{q}(t)), $$
	where $f(t,x,y)=g(t,x)-\frac{R}{L}y$ and $g(t,x)=\frac{1}{L}V(t)-\frac{1}{LC}x$. \\
	Hence, with initial conditions $q(0)=\dot{q}(0)=0$, when the function $V$ is subject to perturbation, our result in Theorem \ref{thm-stab-RLC} can be applied whenever the perturbed form of the function $g$ satisfies all the assumptions of that Theorem.
\end{rmk}

\subsection{Linear control systems}
In this subsection, we give a further interesting application of our main result in Theorem \ref{main-th}.
%in the previous section
In this respect, for given 2$\times$2 matrices $A, B, C, D$, a nonnegative real number $\gamma$ and an initial condition $(x_{0},\dot{x}_{0})$, we aim to tackle linear control systems of second order ODE of the form:
%we consider linear control systems of the form:
\begin{equation}\label{LCS}
\tag{$LCS$}  \; \begin{cases}
\ddot{x}(t)=Ax(t)+\gamma \dot{x}(t)+ B u(t), \\
z(t)=C(x(t)+\dot{x}(t))+Du(t),  \\
x(0)=x_{0},\,\,\, \dot{x}(0)=\dot{x}_{0},\,\,\, z(0)= C(x_{0}+\dot{x}_{0})+Du(0),\\
\end{cases}
\end{equation}
where $x(.), z(.)$ and $u(.)$ are respectively the state variable,  the observability variable and the control of the system. In this example, the function $f$ is nothing else but
$$ f(t,x(t),y(t))=Ax(t)+\gamma y(t)+ B u(t).$$
We consider here a linear perturbation of the matrix $A$, i.e., $A=(a_{i j})_{i,j=1,2}$ and
$$ A_{\lambda}=\begin{pmatrix}
a_{11}+\lambda-\overline{\lambda} & a_{12} \\
a_{21} & a_{21}+\lambda-\overline{\lambda}
\end{pmatrix}, $$
$\overline{\lambda}$ being the initial value of $\lambda$.
Therefore, the perturbed linear control system under consideration is stated as follows:
%define the perturbed problem as follows:
%In this way, the perturbed linear control system we consider is as follows:
\begin{equation}\label{pert-lcs}
\tag{$LCS_{\lambda}$}  \;
\begin{cases}
\ddot{x}_{\lambda}(t)=A_{\lambda}x(t)+\gamma \dot{x}_{\lambda}(t)+ B_{\lambda} u(t), \\
z_{\lambda}(t)=C(x_{\lambda}(t)+\dot{x}_{\lambda}(t))+D_{\lambda}u(t),  \\
x_{\lambda}(0)=x_{0,\lambda},\,\,\, \dot{x}_{\lambda}(0)=\dot{x}_{0,\lambda},\,\,\, z_{\lambda}(0)= C(x_{0,\lambda}+\dot{x}_{0,\lambda})+D_{\lambda}u(0),\\
\end{cases}
\end{equation}
and the associated function $f_{\lambda}$ is given by
$$ f_{\lambda}(t,x(t),y(t))=A_{\lambda}x(t)+\gamma y(t)+B_{\lambda}u(t).$$

In the next, we will make use of the following assumptions:
\begin{itemize}
	\item[$(H_1)$] there exist $r>0$ and a neighborhood $\mathcal{V}(\overline{\lambda})$ of $\overline{\lambda}$ such that for all $\lambda \in \mathcal{V}(\overline{\lambda})$, $x,\, x_{\lambda} \in B(x_0,r)\cap B(x_{0,\lambda},r)$ and  $x^{\prime},\, x_{\lambda}^{\prime} \in B(\dot{x}_{0},r)\cap B(\dot{x}_{0,\lambda},r)$;
	\item[$(H_2)$]
	it is further assumed that on the same neighborhood $\mathcal{V}(\overline{\lambda})$ the norm $\|A_{\lambda}\|$ is independently bounded in $\lambda$, that is there exists $\alpha >0 $ such that $\|A_{\lambda}\|\leq \alpha$;
	\item[$(H_3)$] there exists a constant $\beta>0$ such that for all $\lambda,\lambda^{\prime}\in \mathcal{V}(\overline{\lambda})$ we have,
	$$\|B_{\lambda}-B_{\lambda^{\prime}}\|\leq \beta \| \lambda-\lambda^{\prime}\|\quad\text{and}\quad \|D_{\lambda}-D_{\lambda^{\prime}}\|\leq \beta \| \lambda-\lambda^{\prime}\|;$$
	\item[$(H_4)$] the control $u$ is such that $\|u\|_{\infty}\leq 1$.
\end{itemize}

In the Banach space $ C([0,T];H)$, we consider the norm $\|\,.\,\|_{\infty}$ given for a function $x$ by $\|x\|_{\infty}=\displaystyle \sup_{t\in[0,T]}\|x(t)\|$ and
state the following stability result for \eqref{pert-lcs}:
\begin{thm}\label{appcont}
	Under the assumptions above, the following estimation holds:
	\begin{equation}\label{z-lambda}
	\begin{split}
	\|z-z_{\lambda}\|_{\infty}\leq  c_1\|x_{0,\lambda}-x_{0}\|
	+c_2\|\dot{x}_{0,\lambda}-\dot{x}_{0}\|+c_3\|\lambda-\overline{\lambda}\|,
	\end{split}
	\end{equation}
	where, $$c_1=\|C\|\left(\frac{2+LT-2L}{2+LT}+\frac{2L(4+LT)}{(2+LT)^2}e^{\frac{(2+LT)T}{2}}\right),\;\; c_2=\frac{\|C\|(4+LT)}{2+LT}e^{\frac{(2+LT)T}{2}} $$  and $$c_3=\beta+\frac{4\|C\|L^{\prime}(1-L^{\prime})}{(2+LT)^2}e^{\frac{(2+LT)T}{2}}, \;\;\text{with}\;\; L=\sup(\alpha,\gamma) \;\; \text{and} \;\; L^{\prime}=r+\|x_0\|+\beta.$$
\end{thm}

\begin{proof}
%Then, estimation \eqref{z-lambda}
We have for all $\lambda \in \mathcal{V}(\overline{\lambda})$,
$$\|z-z_{\lambda}\|_{\infty}\leq \|C\|(\|x-x_{\lambda}\|_{\infty}+\|\dot{x}-\dot{x}_{\lambda}\|_{\infty})+\|D-D_{\lambda}\|\|u\|_{\infty}.$$
Then, by assumptions $(H_3)$ and $(H_4)$, we get
\begin{equation}\label{z-zlambda}
\|z-z_{\lambda}\|_{\infty}\leq \|C\|(\|x-x_{\lambda}\|_{\infty}+\|\dot{x}-\dot{x}_{\lambda}\|_{\infty})+\beta\|\lambda-\overline{\lambda}\|.
\end{equation}
Clearly, assumption $(H_1)$  implies that
$$\|A_{\lambda}x-A_{\lambda^{\prime}}x\|\leq (r+\|x_0\|)\|\lambda-\lambda^{\prime}\|.$$
Combining this inequality with hypotheses $(H_3)$ and $(H_4)$, we conclude that $f_\lambda$ is $L'=(r+\|x_0\|+\beta)$-Lipschitz in $\lambda$ uniformly in $t$ and $(x,y)$. \\
On the other hand, under assumption $(H_2)$, the function $f_\lambda$ is $L$-Lipschitz in $(x,y)$ uniformly in $t$ and $\lambda$ with $L=\sup(\alpha,\gamma)$. \\
Consequently, assumptions of Theorem \ref{main-th} are all satisfied. Applying this theorem, immediately follows from the estimate \eqref{estip}:
\begin{equation*}\label{}
\begin{split}
\|x-x_{\lambda}\|_{\infty}\leq&  \left(1+L\left(\frac{2}{2+LT}\right)^{2}e^{\frac{(2+LT)T}{2}}\right)\|x_{0,\lambda}-x_{0}\| \\
&+\frac{2}{2+LT}e^{\frac{(2+LT)T}{2}}\|\dot{x}_{0,\lambda}-\dot{x}_{0}\|+
L^{\prime}\left(\frac{2}{2+LT}\right)^{2}e^{\frac{(2+LT)T}{2}}\|\lambda-\overline{\lambda}\|.
\end{split}
\end{equation*}
Using thus inequality \eqref{xlambda'}, we also have
$$\|\dot{x}_{\lambda}-\dot{x}\|_{\infty}\leq  \|\dot{x}_{0,\lambda}-\dot{x}_{0}\|e^{\frac{(2+LT)T}{2}}+2\frac{L\|x_{0,\lambda}-x_{0}\|+L^{\prime}\|\lambda-\overline{\lambda}\|}{2+LT}( e^{\frac{(2+LT)T}{2}}-1).
$$
Combining the last two inequalities with \eqref{z-zlambda} yields the desired estimate \eqref{z-lambda}.
\end{proof}

%\begin{rmk}
%	In \cite{A}, based on Gronwall's techniques, the authors established a stability result of the observability variable in parametric linear control systems of first order in time. In our purpose, we rather investigate the parametric stability of the second order.  Accordingly, our result may be seen as an extension and refinement of that presented in \cite{A}. 
%\end{rmk}

\section{Numerical experiment}

In this section, we present numerical tests of our stability based on the  Scientific Laboratory, Scilab,  programming language, so we take back to the linear control system \eqref{LCS} considered in the previous section. The objective here is to justify the validity of Theorem \ref{appcont} when the parametric functions satisfy the conditions of that theorem. In this regard, for the sake of simplicity, we make the following standard choices:
\begin{center}
	$A=\begin{pmatrix}
	0 & -3 \\
	1 & -4
	\end{pmatrix}$ \quad
	and \quad $C=\begin{pmatrix}
	1&  1 \\
	1 & 1
	\end{pmatrix}$.	
\end{center}
We take the matrix $B=D=0$, the control $u(t)=\left(
\begin{array}{c}
1 \\
1 \\
\end{array}\right)$                                                  doesn't dependent on time and start from the following initial conditions $x_{0}=\left(
\begin{array}{c}
1 \\
1 \\
\end{array}
\right)$
and $\dot{x}_{0}=\left(
\begin{array}{c}
0 \\
1 \\
\end{array}
\right).$
%The functions $x(.)$, $z(.)$ are respectively the state variable and the observability variable of the system.
\\
\noindent
For the associated parametric form \eqref{pert-lcs} of our example,
we consider here a linear perturbation of the matrices $A,B$ and $D$ as follows:
$$A_{\lambda}=\begin{pmatrix}
\lambda & -3 \\
1 & -4+\lambda
\end{pmatrix}\quad \text{and} \quad  B_{\lambda}=D_{\lambda}=\begin{pmatrix}
\lambda & 0 \\
0 & \lambda
\end{pmatrix},
$$
$\overline{\lambda}=0$  being the initial value of $\lambda$, i.e.,  $x_{\overline{\lambda}}=x,\, A_{\overline{\lambda}}=A,\, B_{\overline{\lambda}}=B=0,\, C_{\overline{\lambda}}=C$\,and $D_{\overline{\lambda}}=D=0.$ \\
%Therefore, the perturbed linear control system under consideration is stated as follows:

%\begin{equation}\label{pert-lcs}
%\tag{$LCS_{\lambda}$}  \;
%\begin{cases}
% x_{\lambda}^{\prime\prime}(t)=A_{\lambda}x(t)+\gamma x_{\lambda}^{\prime}(t)+ B_{\lambda} u(t), \\
% z_{\lambda}(t)=C(x_{\lambda}(t)+x_{\lambda}^{\prime}(t))+D_{\lambda}u(t),  \\
% x_{\lambda}(0)=x_{0,\lambda},\,\,\, x_{\lambda}^{\prime}(0)=x_{0,\lambda}^{\prime},\,\,\, z_{\lambda}(0)= C(x_{0,\lambda}+x_{0,\lambda}^{\prime})+D_{\lambda}u(0),\\
%\end{cases}
%\end{equation}
%The objective of this section is to justify the validity of Theorem \ref{appcont} since the parametric functions satisfy the conditions of the cited theorem.

Figure 3 below describes the behavior of the norm of the difference between the observability solutions: $z_{\lambda}$ of the parametric system \eqref{pert-lcs} and $z$ of the initial one \eqref{LCS}, i.e., the solution for $\overline\lambda=0.$
%\medskip

%\begin{minipage}{0.5\textwidth}
%   \begin{figure}[H]
%      \includegraphics[scale=0.42]{Fig3.pdf}
%     \caption{\small The uniform convergence of $z_{\lambda}$ to $z$ \newline with not perturbed intial condition.}
% \end{figure}
%\end{minipage}
%\hspace{2ex} % eventuellement
%\begin{minipage}{0.5\textwidth}
%   \begin{figure}[H]
%      \includegraphics[scale=0.42]{Fig4.pdf}
%     \caption{\small The uniform convergence of $z_{\lambda}$ to $z$ \newline with perturbed intial condition.}
%\end{figure}
%\end{minipage}

\begin{flushleft}
	\begin{figure}[h]
		%\subfloat[]
		{\includegraphics[scale=0.36]{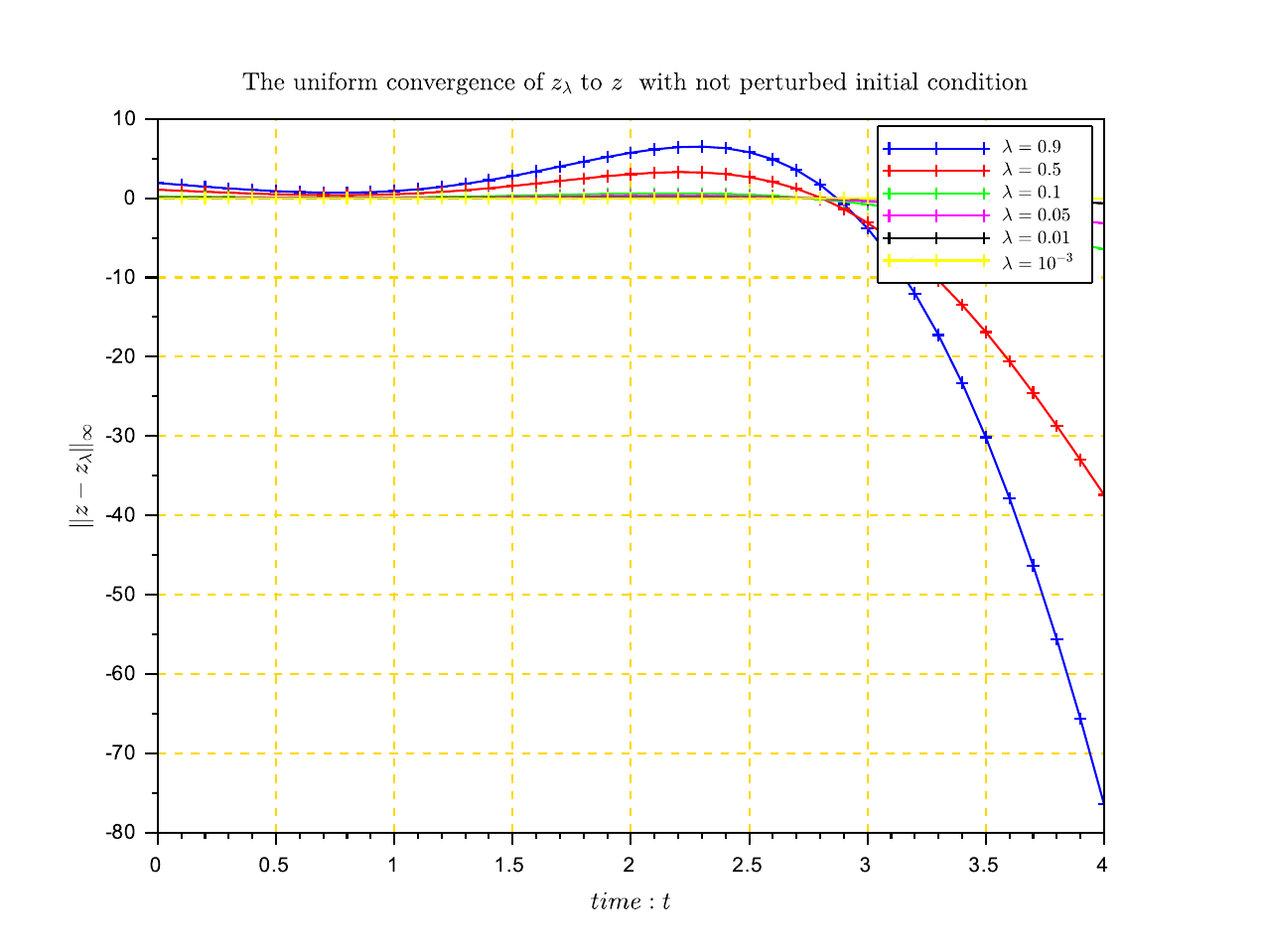}}
		%\subfloat[]
		{\includegraphics[scale=0.36]{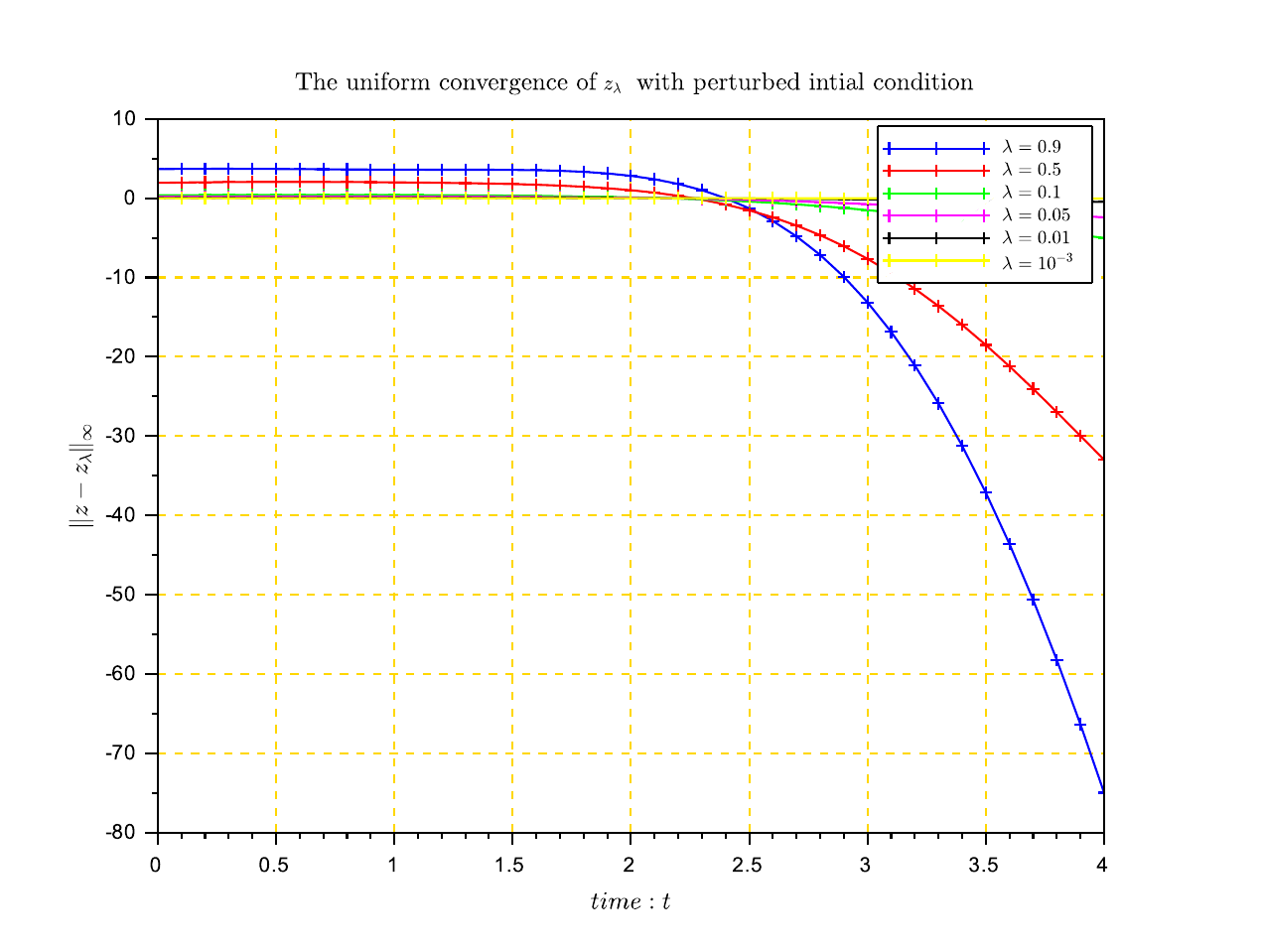}}
		% Un titre pour les deux figures
		\caption{\small The uniform convergence of $z_{\lambda}$ to $z$.}
	\end{figure}
\end{flushleft}
%\medskip
%\noindent

The numerical results in the left hand side of Figure 3 illustrate the rate of convergence of $\|z-z_{\lambda}\|_{\infty}$
%(the norm of the difference of observability solutions)
for various choices of $\lambda$, when the initial condition is not subject to perturbation, while the second panel displays the convergence rate of $\|z-z_{\lambda}\|_{\infty}$ for the perturbed system starting from perturbed initial condition $(x_{0,\lambda},\dot{x}_{\lambda,0})$ with:
$x_{0,\lambda}=\left(
\begin{array}{c}
1+\lambda \\
1+\lambda \\
\end{array}
\right)$
and $\dot{x}_{\lambda,0}=\dot{x}_{0}$ is keeping constant, by using the same values of the perturbation parameter $\lambda$.
Clearly, both plots indicate that the norm $\|z-z_{\lambda}\|_{\infty}$  decreases linearly with respect to the parameter $\lambda$ which confirm the validity of the inequality \eqref{z-lambda}. \\
%\begin{itemize}
%\item[$\bullet$]
%In Figure 3 above we check the validity of the inequality (\ref{z-lambda}) when the initial condition is not perturbed.
%This clearly indicates that the norm $\|z-z_{\lambda}\|_{\infty}$ (of the difference of the observability solutions) decreases linearly with respect to the parameter $\lambda$.
%\item [$\bullet$]
%In Figure 4, we consider the perturbation of the initial condition $x_0$ as follows:
%$x_{0,\lambda}=\left(
%                                  \begin{array}{c}
%                                   1+\lambda \\
%                                  1+\lambda \\
%                               \end{array}
%                            \right)$
%by keeping  $x^{\prime}_{\lambda,0}=x^{\prime}_{0}$ constant, as shown in the Figure 4,
%we can make the same observation. \\
%\end{itemize}

\noindent
{\bf Acknowledgements}.
The authors would like to thank Prof. M. Ait Mansour for the nice discussion on the topic and the valuable comments on the content of the paper.
The second author would like to thank his supervisor, Prof. J. Lahrache, for the continuous support during the preparation of his thesis. The authors wish equally to warmly thank the editor in chief of ASVAO Journal Prof. Christiane Tammer, and as well by the meantime, the guest editors Profs: Valentina Sessa, Elisabeth K\"{o}bis, Jaafar Lahrache and Mohamed Ait Mansour for their very nice cooperation during the workshops ASEM 23-El Jadida and ASSVAS 23-Safi to which is dedicated this issue.

%\newpage

\end{document}